\documentclass[a4paper,twopage,reqno,11pt]{amsart}

\usepackage{pdfsync}
\usepackage{graphicx}
\usepackage{amsmath}
\usepackage{amssymb}
\usepackage{amsfonts}
\usepackage{amsthm}
\usepackage{amstext}
\usepackage{amsbsy}
\usepackage{amsopn}
\usepackage{amscd}
\usepackage{enumerate}
\usepackage{color}
\usepackage{url}
\usepackage[colorlinks]{hyperref}
\usepackage[hyperpageref]{backref}
\usepackage{multirow}
\usepackage{lipsum}
\usepackage{rotating}
\usepackage{longtable}
\usepackage[]{caption}
\usepackage{float}
\usepackage{lscape}
\usepackage{pdflscape}
\usepackage[]{algorithm2e}
\usepackage{color}
\usepackage[colorlinks]{hyperref}
\hypersetup{
    colorlinks=true,%
    citecolor=green,%
    filecolor=black,%
    linkcolor=blue,%
    urlcolor=magenta
}

\newtheorem{theorem}{Theorem}[section]
\newtheorem{corollary}[theorem]{Corollary}

\newtheorem{lemma}[theorem]{Lemma}

\theoremstyle{definition}

\numberwithin{equation}{section}

\newcommand{\PSL}{\mathrm{PSL}}
\newcommand{\PSU}{\mathrm{PSU}}

\newcommand{\POm}{\mathrm{P \Omega}}
\newcommand{\J}{\mathrm{J}}

\newcommand{\A}{\mathrm{A}}

\renewcommand{\S}{\mathrm{S}}

\newcommand{\D}{\mathrm{D}}

\newcommand{\nr}{\mathrm{nr}}
\renewcommand{\i}{\mathrm{i}}
\newcommand{\He}{\mathrm{He}}
\newcommand{\McL}{\mathrm{McL}}
\newcommand{\Co}{\mathrm{Co}}
\newcommand{\HN}{\mathrm{HN}}
\newcommand{\ON}{\mathrm{O'N}}
\newcommand{\Suz}{\mathrm{Suz}}
\newcommand{\Fi}{\mathrm{Fi}}
\newcommand{\HS}{\mathrm{HS}}
\newcommand{\Ly}{\mathrm{Ly}}
\newcommand{\Th}{\mathrm{Th}}
\newcommand{\Ru}{\mathrm{Ru}}

\newcommand{\Aut}{\mathrm{Aut}}

\newcommand{\B}{\mathrm{B}}

\newcommand{\M}{\mathrm{M}}

\newcommand{\Dmc}{\mathcal{D}}
\newcommand{\Bmc}{\mathcal{B}}

\newcommand{\Pmc}{\mathcal{P}}
\newcommand{\Cmc}{\mathcal{C}}

\renewcommand{\leq}{\leqslant}
\renewcommand{\geq}{\geqslant}

\renewcommand{\mod}[1]{\ (\mathrm{mod}{\ #1})}
\newcommand{\imod}[1]{\allowbreak\mkern4mu({\operator@font mod}\,\,#1)}

\begin{document}
 \title[]{Sporadic simple groups as flag-transitive  automorphism groups of symmetric designs}

 \author[S.H. Alavi]{Seyed Hassan Alavi}
 \thanks{Corresponding author: S.H. Alavi}
 \address{Seyed Hassan Alavi, Department of Mathematics, Faculty of Science, Bu-Ali Sina University, Hamedan, Iran.}
 \email{alavi.s.hassan@basu.ac.ir and  alavi.s.hassan@gmail.com}
 \author[A. Daneshkhah]{Ashraf Daneshkhah}
\address{Ashraf Daneshkhah, Department of Mathematics, Faculty of Science, Bu-Ali Sina University, Hamedan, Iran.}
\email{adanesh@basu.ac.ir and  daneshkhah.ashraf@gmail.com}

 \date{\today}%

\begin{abstract}
In this article, we study symmetric designs admitting flag-transitive, point-imprimitive almost simple automorphism groups with socle sporadic simple groups. As a corollary, we present a classification of symmetric  designs admitting flag-transitive automorphism group whose socle is a sporadic simple group, and in conclusion, there are exactly seven such designs, one of which admits a point-imprimitive automorphism group and the remaining are point-primitive. 
\end{abstract}

\maketitle

\section{Introduction}\label{sec:intro}

A \emph{symmetric $(v,k,\lambda)$ design} is an incidence structure $\Dmc=(\Pmc,\Bmc)$  consisting of a set $\Pmc$ of $v$ \emph{points} and a set $\Bmc$ of $v$ \emph{blocks} such that every point is incident with exactly $k$ blocks, and each pair of blocks is incident with exactly $\lambda$ points. If $2<k<v-1$, then  $\Dmc$ is called a \emph{nontrivial} symmetric design.
The complement of $\Dmc$ is a symmetric design whose points are the points of $\Dmc$ and blocks the complements of the blocks of $\Dmc$.
A \emph{flag} of $\Dmc$ is an incident pair $(\alpha,B)$, where $\alpha$ and $B$ are a point and a block of $\Dmc$, respectively. 
An \emph{automorphism} of a symmetric design $\Dmc$ is a permutation of the points permuting the blocks and preserving the incidence relation. The automorphism group $G$ of $\Dmc$ is called \emph{flag-transitive} if it is transitive on the set of flags of $\Dmc$, and $G$ is called \emph{anti-flag-transitive} if it is flag-transitive on the complement of $\Dmc$. If $G$ preserves a nontrivial $G$-invariant partitions on  the point set $\Pmc$, then $G$ is said to be \emph{point-imprimitive}, otherwise, $G$ is called point-primitive. We use the standard notation for the groups in \cite{b:Beth-I,b:Atlas,b:KL-90}.

The sporadic simple groups as one of the classes of finite simple groups have always been thought as an important tool to study the symmetries of geometries. For example, Witt designs as the Steiner systems are better understood by their automorphism groups which are the sporadic simple Mathieu groups discovered 70 years earlier, see \cite[Chapter IV]{b:Beth-I}. This paper is devoted to studying the sporadic almost simple groups $G$ as flag-transitive automorphism groups of symmetric designs. This problem have been studied for the point-primitive automorphism groups  \cite{a:Zhou-sym-sporadic}. Here, we are interested in the case where $G$ is point-imprimitive. 

\begin{theorem}\label{thm:main}
Let $\Dmc$ be a nontrivial symmetric design admitting a flag-transitive automorphism group $G$ with socle a sporadic simple group. Then $G$ is point-imprimitive if and only if  $\Dmc$ has parameter set $(144,66,30)$, $G=\M_{12}$, and the point-stabiliser and the block-stabiliser are isomorphic to $\PSL_{2}(11)$, and both are the intersection of two non-conjugate maximal subgroups of $G$ isomorphic to $\M_{11}$. 
\end{theorem}

In \cite{a:Praeger-imprimitive}, Praeger and Zhou  studied point-imprimitive symmetric $(v,k,\lambda)$ designs, and they improved the upper bound of $k$ (and so $v$) in terms of $\lambda$ given by O’Reilly-Regueiro  \cite{a:Regueiro-reduction}, and a list of possible parameters for $\lambda\leq 10$ is given some of which is still open. This motivates Law, Praeger and Reichard \cite{a:Praeger-96-20-4} to classify flag-transitive symmetric $(96, 20, 4)$ designs, and as a result of their work,  all symmetric designs with $\lambda\leq 4$ admitting a flag-transitive, point-imprimitive
subgroup of automorphisms are known. In 2022, 
Mandi\'{c}  and \v{S}uba\v{s}i\'{c} \cite{a:Mandic-Sym-Imp-lam10} have treated all other open possibilities reported in \cite{a:Praeger-imprimitive} excluding two sets of parameters. Recently, Montinaro \cite{a:Montinaro-sym-imp} made a contribution to the study of point-imprimitive symmetric designs when $k>\lambda(\lambda-3)/2$ with size of the intersection of block and imprimitivity class at least $3$. In our study of point-imprimitive symmetric designs with sporadic almost simple automorphism groups, we observe that the candidate parameter sets fall into type ``a'' or ``b'', the  first two parts in \cite[Theorem 1.1]{a:Praeger-imprimitive}. The outline and the method of proving Theorem \ref{thm:main} is described in Section~\ref{sec:method}. 

By Theorem \ref{thm:main} and the main result of \cite{a:Zhou-sym-sporadic}, we are now able to present a classification result for symmetric designs admitting flag-transitive sporadic almost simple automorphism groups. The detailed information about these designs are given in Section \ref{sec:ex}.

\begin{corollary}\label{cor:main}
Let $\Dmc$ be a nontrivial symmetric design admitting a flag-transitive automorphism group with socle a sporadic simple group. Then $v\in\{144,176,14080\}$ and $(\Dmc,G)$ is (up to isomorphism) as one of the rows in {\rm Table~\ref{tbl:main}}. 
\end{corollary}
\begin{table}[h]
\scriptsize
\caption{Sporadic simple groups and symmetric $2$-designs.}\label{tbl:main}
\begin{tabular}{clllllllllll}
\noalign{\smallskip}\hline\noalign{\smallskip}
Line &
Design &
$v$ &
$k$ &
$\lambda$ &
$G$ &
$G_{\alpha}$ &
$G_{B}$ &
Comments & 
References \\
\noalign{\smallskip}\hline\noalign{\smallskip}
$1$ &
$\Dmc_{1}$ & 
$144$ & $66$ & $30$ &
$\M_{12}$ &
$\PSL_{2}(11)$ &
$\PSL_{2}(11)$ & 
point-imprimitive &
Theorem 1.1 \\
$2$ &
$\Dmc_{2}$ & 
$144$ & $66$ & $30$ &
$\M_{12}$, $\M_{12}{:2}$ &
$\PSL_{2}(11)$ &
$\PSL_{2}(11)$ & 
point-primitive &
\cite{a:Zhou-sym-sporadic} \\
$3$ &
   $\Dmc_{3}$ & 
   $144$ & $66$ & $30$ &
$\M_{12}:2$ & 
$\PSL_{2}(11):2$ &
$\PSL_{2}(11):2$ & 
point-primitive &
\cite{a:Zhou-sym-sporadic} \\
$4$ &
$\Dmc_{4}$ & 
$176$ & $126$ & $90$ &
$\M_{22}$ &
$\A_7$& 
$\A_{7}$ &  
point-primitive &
\cite{a:Dempwolff2001,a:Zhou-sym-sporadic} \\
$5$ &
$\Dmc_{5}$ & 
$176$ & $50$ & $14$ &
$\HS$ & 
$\PSU_3(5):2$ & 
$\PSU_3(5):2$ & 
point-primitive &
\cite{a:Kantor-85-2-trans, a:Zhou-sym-sporadic} \\
$6$ &
$\Dmc_{6}$ & 
$176$ & $126$ & $90$ &
$\HS$ & 
$\PSU_3(5):2$ & 
$\PSU_3(5):2$ & 
point-primitive &
\cite{a:Dempwolff2001,a:Zhou-sym-sporadic} \\
$7$ &
$\Dmc_{7}$ & 
$14080$ & $12636$ & $11340$ &
$\Fi_{22}$ &
$\POm_7(3)$ & 
$\POm_7(3)$&  
point-primitive &
\cite{a:Dempwolff2001,a:Zhou-sym-sporadic} \\
\noalign{\smallskip}\hline\noalign{\smallskip}
\multicolumn{10}{l}{Note: $G_{\alpha}$ is the  point-stabiliser of $\alpha$ and $G_{B}$ is the  block-stabiliser of $B$.}
\end{tabular}
\end{table}

\section{Examples and constructions}\label{sec:ex}

Here, we give the examples of symmetric $(v, k, \lambda)$ designs admitting flag-transitive almost simple automorphism groups with socle sporadic simple groups. The point-primitive designs (Lines 2-7 of Table \ref{tbl:main}) are obtained in \cite{a:Zhou-sym-sporadic}, see also \cite{a:Dempwolff2001,a:Kantor-75-2-trans}. The point-imprimitive design (Line 1 of Table \ref{tbl:main}) is obtained in Theorem~\ref{thm:main} and is constructed in Section \ref{sec:proof}. We use GAP \cite{GAP4} for the computations and in particular we use the software package ``\verb|design|'' in GAP  for design constructions. In what follows,  ``$\nr(G)$'' is the position of the primitive group $G$  in the list of the library of the primitive permutation groups in GAP. Moreover, the subgroups $H$ and $K$ are the point-stabiliser of a point and the block-stabiliser of a block of $\Dmc$, respectively.\smallskip

\noindent \textbf{Line 1.}
The design $\Dmc_{1}$ in this line is point-imprimitive with $\M_{12}$ as its automorphism  group. Here, we follow the proof of Theorem \ref{thm:main}, and present an explicit construction of this design. Let $G$, $H$ and $K$ be the groups generated by the permutations on $144$ points listed in Table \ref{tbl:gens}. Then $G$ is the permutation group on $144$ points which is isomorphic to $\M_{12}$, and $H$ and $K$ are both subgroups of $G$ of index $144$. The subgroups $H$ and $K$ are conjugate in $G$ which are isomorphic to $\PSL_{2}(11)$. These subgroups are not maximal, and they are in fact contained in maximal subgroups isomorphic to $\M_{11}$. The group $G$ is point-imprimitive with point-stabiliser $H$ fixing $1$. The subgroup $K$ has nontrivial orbits of length  $11$,
$11$, $55$ and $66$. The $K$-orbit $B$ given in \eqref{eq:D-M12-144} has length $66$ and it is a base block for the unique design $\Dmc_{1}$. As it is shown in the proof of Theorem \ref{thm:main}, the only possibility for a subgroup to be a block-stabiliser of index $12$ is when it is a subgroup of one of two non-conjugate maximal subgroups  (isomorphic to) $\M_{11}$. For any pairs of such subgroups with orbit length $66$, the associated symmetric designs are isomorphic. The full automorphism group of this design is $\M_{12}:2$ which is point-primitive, block-primitive and flag-transitive but not anti-flag-transitive. The same is true for $G=\M_{12}$ and the group $G$ has rank $5$ on its action on $144$ points. 
We note that the permutation representation of $G$ on $144$ points that we have given here is obtained by the coset action of $G$ on the set of right cosets of the intersection of two non-conjugate maximal subgroups isomorphic to $\M_{11}$. Indeed, we start with the permutation representation of $\M_{12}$ on $12$ points, and once we obtain the subgroups $H$ and $K$ as explained in the proof of Theorem \ref{thm:main}, we consider the right coset action of $\M_{12}$ on $144$ right cosets of $H$. We then find the presentations of $G$, $H$ and $K$ as in Table \ref{tbl:gens} and construct the design as described. \smallskip

\begin{table}[h]
\scriptsize
\caption{The generators of the group $G$ and its subgroups $H$ and $K$.}\label{tbl:gens}
\begin{tabular}{cp{5mm}p{140mm}}
\noalign{\smallskip}\hline\noalign{\smallskip}
Group & Generators\\
\noalign{\smallskip}\hline\noalign{\smallskip}%
 $G$ & $\alpha_{1}=$& (2,11)(3,6)(4,12)(9,10)(13,109)(14,119)(15,120)(16,114)(17,116)(18,112)(19,115)(20,113)(21,117)(22,118)
 (23,110)(24,111)(25,30)(26,29)(28,31)(33,34)(37,95)(38,91)(39,93)(40,90)(41,96)(42,86)(43,85)(44,89)(45,92)
 (46,94)(47,87)(48,88)(49,54)(50,59)(55,60)(56,58)(61,129)(62,128)(63,130)(64,124)(65,131)(66,132)(67,125)
 (68,127)(69,122)(70,126)(71,123)(72,121)(73,78)(74,83)(79,84)(80,82)(97,137)(98,144)(99,135)(100,133)
 (101,139)(102,136)(103,140)(104,134)(105,142)(106,138)(107,141)(108,143)\\
&  
$\alpha_{2}=$& (1,97,132)(2,107,128)(3,102,126)(4,108,123)(5,101,121)(6,99,129)(7,103,125)(8,104,127)(9,106,122)
(10,105,130)(11,98,124)(12,100,131)(13,78,36)(14,74,27)(15,83,26)(16,82,29)(17,80,28)(18,77,30)(19,79,32)
(20,84,31)(21,75,33)(22,81,34)(23,76,35)(24,73,25)(37,96,137)(38,92,141)(39,90,136)(40,87,143)(41,85,139)
(42,93,135)(43,89,140)(44,91,134)(45,86,138)(46,94,142)(47,88,144)(48,95,133)(49,61,111)(50,71,120)
(51,68,116)(52,62,119)(53,67,115)(54,66,109)(55,72,113)(56,70,114)(57,69,118)(58,63,117)(59,64,110)
(60,65,112)\\
$H$ & $\beta_{1}=$& (2,11)(3,6)(4,12)(9,10)(13,109)(14,119)(15,120)(16,114)(17,116)(18,112)(19,115)(20,113)(21,117)(22,118)
(23,110)(24,111)(25,30)(26,29)(28,31)(33,34)(37,95)(38,91)(39,93)(40,90)(41,96)(42,86)(43,85)(44,89)(45,92)
(46,94)(47,87)(48,88)(49,54)(50,59)(55,60)(56,58)(61,129)(62,128)(63,130)(64,124)(65,131)(66,132)(67,125)
(68,127)(69,122)(70,126)(71,123)(72,121)(73,78)(74,83)(79,84)(80,82)(97,137)(98,144)(99,135)(100,133)
(101,139)(102,136)(103,140)(104,134)(105,142)(106,138)(107,141)(108,143)\\
& $\beta_{2}=$& (3,6)(4,7)(5,11)(8,10)(13,128)(14,130)(15,122)(16,132)(17,121)(18,12)(19,124)(20,126)(21,127)(22,131)(23,125)
(24,129)(26,29)(27,30)(31,36)(32,33)(37,115)(38,109)(39,116)(40,114)(41,117)(42,111)(43,119)(44,113)(45,120)
(46,118)(47,112)(48,110)(49,51)(50,53)(52,60)(56,58)(61,86)(62,91)(63,85)(64,95)(65,94)(66,90)(67,88)(68,96)
(69,92)(70,89)(71,87)(72,93)(73,135)(74,139)(75,144)(76,143)(77,137)(78,134)(79,142)(80,141)(81,140)(82,136)
(83,138)(84,133)(99,104)(100,105)(101,106)(102,107)\\
&
$\beta_{3}=$& (2,9)(3,5)(4,11)(10,12)(14,16)(15,17)(19,20)(21,23)(25,69)(26,71)(27,
64)(28,68)(29,63)(30,65)(31,67)(32,72)
(33,70)(34,61)(35,62)(36,66)(37,54)(38,57)(39,51)(40,50)(41,53)(42,55)(43,59)(44,52)(45,60)(46,49)(47,58)
(48,56)(73,113)(74,119)(75,111)(76,118)(77,112)(78,115)(79,116)(80,110)(81,109)(82,117)(83,114)(84,120)
(85,90)(86,92)(87,88)(94,95)(97,139)(98,134)(99,142)(100,135)(101,136)(102,137)(103,140)(104,143)(105,133)
(106,138)(107,141)(108,144)(122,131)(123,130)(125,127)(126,129)\\
$K$ & $\gamma_{1}=$ & (1,106)(2,97)(3,108)(4,102)(5,98)(6,104)(7,107)(8,103)(9,99)(10,101)(11,105)(12,100)(13,64)(14,71)(15,61)
(16,70)(17,65)(18,67)(19,68)(20,72)(21,63)(22,66)(23,69)(24,62)(25,120)(26,110)(27,109)(28,112)(29,114)
(30,117)(31,113)(32,115)(33,116)(34,119)(35,111)(36,118)(37,48)(38,41)(39,
40)(45,47)(49,58)(51,59)(53,57)
(55,56)(73,121)(74,130)(75,128)(76,132)(77,127)(78,125)(79,122)(80,131)(81,124)(82,123)(83,129)(84,126)
(85,93)(86,88)(87,94)(91,96)(133,137)(135,139)(136,142)(138,140)\\
& 
$\gamma_{2}=$& (1,136)(2,134)(3,135)(4,143)(5,142)(6,133)(7,137)(8,139)(9,138)(10,144)(11,
141)(12,140)(13,20)(14,22)(15,17)
(21,24)(25,111)(26,116)(27,115)(28,109)(29,120)(30,117)(31,113)(32,112)(33,118)(34,119)(35,114)(36,110)
(37,75)(38,82)(39,77)(40,81)(41,74)(42,83)(43,78)(44,73)(45,76)(46,80)(47,79)(48,84)(49,66)(50,68)(51,72)
(52,67)(53,61)(54,69)(55,62)(56,63)(57,
65)(58,71)(59,64)(60,70)(85,89)(86,96)(87,88)(91,94)(98,106)(99,100)
(103,
108)(104,107)(122,132)(123,130)(124,127)(126,128)
\\
& 
$\gamma_{3}=$& (1,11)(3,4)(6,10)(9,12)(13,110)(14,109)(15,120)(16,112)(17,113)(18,117)(19,116)(20,115)(21,114)(22,118)
(23,119)(24,111)(25,40)(26,42)(27,41)(28,
37)(29,46)(30,47)(31,39)(32,48)(33,38)(34,43)(35,45)(36,44)(49,128)
(50,121)(51,130)(52,132)(53,124)(54,125)(55,127)(56,131)(57,122)(58,129)(59,126)(60,123)(61,65)(62,67)
(68,71)(70,72)(73,103)(74,108)(75,106)(76,
107)(77,99)(78,97)(79,101)(80,100)(81,104)(82,98)(83,105)(84,102)
(85,
91)(86,94)(87,90)(93,95)(133,144)(136,139)(137,140)(142,143)\\
\noalign{\smallskip}\hline\noalign{\smallskip}
\end{tabular}
\end{table}

\noindent \textbf{Line 2.} This design is one of the  two point-primitive symmetric designs on $144$ points. The full-automorphism group of this design is $\M_{12}:2$. The groups $\M_{12}$ and $\M_{12}:2$ are  flag-transitive and point-primitive. Let $G=\M_{12}$. The point-stabiliser is the maximal subgroup isomorphic to $\PSL_{2}(11)$. This symmetric design was first constructed by Wirth \cite{t:Wirth}. This design is denoted by $\Dmc_{1}$ in~\cite{a:Zhou-sym-sporadic}. Here $\nr(G)=3$, that is to say, $G=\M_{12}$ is the $3$rd group in the list of primitive permutation groups in GAP. The base block is 
\begin{align*}
B_2=\{&6,7,9,10,13,17,18,19,20,21,32,33,34,36,37,38,46,49,51,52,53,54,58,60,61,62,63,\\
&64,65,66,67,68,69,71,72,74,78,79,81,85,87,89,90,92,94,95,96,98,102,105,106, 109,\\
&110,114,118,121,122,126,127,128,129,132,134,137,141,142\}.
\end{align*}
The point-stabiliser $H$ and the block-stabiliser $K$ of $\Dmc$ are maximal subgroups of $G$, both isomorphic to $\PSL_{2}(11)$ but not conjugate. The full automorphism group of $\Dmc_{2}$ is $\M_{12}:2$ which is point-primitive, block-primitive and flag-transitive, and both $G$ and $\Aut(\Dmc_{2})$ are not anti-flag-transitive.  
\smallskip

\noindent \textbf{Line 3.}
The design in this line is the second symmetric design with parameter set  $(144,66,30)$. Its automorphism group  $\M_{12}{:}2$ is flag-transitive and point-primitive with point-stabiliser (isomorphic to) $\PSL_{2}(11){:}2$. This design is denoted by $\Dmc_{2}$~in \cite{a:Zhou-sym-sporadic}, and it is constructed by the $5$th group $G=\M_{12}:2$ in the list of primitive permutation groups in GAP and the  base block is
\begin{align*}
B_{3}=\{&3,5,6,7,8,10,11,14,16,20,21,23,25,26,27,28,29,33,39,40,44,45,47,49,51,53,54,\\
&
56,57,59,60,61,62,64,65,71,72,73,75,76,79,82,84,85,90,95,96,97,103,104,107, 109,\\
&111,112,113,115,116,119,126,132,134,137,138,142,143,144\}.
\end{align*}
The group $G$ is flag-transitive, point-primitive and block-primitive but not anti-flag-transitive. The point-stabiliser and the block-stabiliser of $\Dmc_{3}$ are isomorphic to $\PSL_{2}(11):2$. We observe that if we take the socle of $G$ and construct the design with the base block $B_{3}$, then we obtain a design which is isomorphic to $\Dmc_{1}$. 
\smallskip

\noindent \textbf{Line 4.} The design in this line arose from studying rank $3$ automorphism groups of symmetric designs \cite{a:Dempwolff2001}. Here $G=\M_{22}$ and $\nr(G)=3$, and if we take the base block $B_{4}$ below, then we obtain $\Dmc_{4}$ as in Line 4 of Table \ref{tbl:main}. This design is denoted by $\Dmc_{3}$~in \cite{a:Zhou-sym-sporadic}. 
\begin{align*}
B_{4}=\{&1, 3, 4, 5, 6, 7, 10, 12, 13, 14, 16, 17, 19, 23, 24, 25, 27, 28, 29, 31, 32, 33, 34, 35, 37, 38, 39,\\
& 40, 41, 42, 43, 44, 45, 46, 
47, 48, 50, 53, 54, 56, 60, 62, 63, 64, 65, 66, 68, 69, 70, 73, 74, 75,\\ 
&76, 77, 78, 79, 80, 81, 82, 83, 84, 85, 86, 88, 89, 90, 91, 92, 
94, 95, 96, 97, 99, 100, 101, 102, \\ 
&103, 105, 106, 107, 108, 109, 112, 114, 117, 119, 120, 123, 124, 125, 126, 128, 129, 132,\\
& 134, 138, 
139, 140, 141, 142, 143, 144, 147, 148, 149, 151, 152, 153, 154, 155, 156, 159, 160,\\
& 161, 162, 163, 164, 165, 167, 168, 171, 172, 173, 
174, 175, 176\}.
\end{align*}
The full automorphism group of $\Dmc_{4}$ is $\HS$   which is point-primitive, block-primitive, flag-transitive and anti-flag-transitive with both  point-stabiliser and block-stabiliser isomorphic to $\PSU_{3}(5){:}2$ but non-conjugate. The group $G=\M_{22}$ is point-primitive, block-primitive, flag-transitive but not anti-flag-transitive, and both point-stabiliser and block-stabiliser are isomorphic to $\A_{7}$.  
\smallskip  

\noindent \textbf{Line 5.} The design in this line is the unique $2$-transitive symmetric $(176,50,14)$ design \cite{a:Kantor-85-2-trans}. This design is denoted by $\Dmc_{4}$ in \cite{a:Zhou-sym-sporadic}. The point-stabiliser (block-stabiliser) of this design is isomorphic to $\PSU_{3}(5){:}2$. The group $G=\HS$ is flag-transitive, point-primitive, block-primitive and anti-flag-transitive, and the base block is 
\begin{align}
\nonumber B_{5}=\{&1,7,10,11,15,17,25,26,30,33,36,40,41,43,51,59,64,66,69,85,90,94,97,\\
&99,106,109,113,114,115,119,122,132,136,142,143,148,150,152,153,154,155,\label{eq:B5}\\
\nonumber & 160,162,164,166, 167,168,169,171,176\}.
\end{align} 
\noindent \textbf{Line 6.} The design in this line is rank $3$ symmetric $(176,126,14)$ design \cite{a:Dempwolff2001} with automorphism group $\HS$. The point-stabiliser (block-stabiliser) is $\PSU_{3}(11)$. The design $\Dmc_{6}$ is the complement of the design in Line 5, and hence the base block of $\Dmc_{6}$ is $\Pmc\setminus B_{5}$, where $\Pmc:=\{1,\ldots,176\}$ and $B_{5}$ is as in \ref{eq:B5}. This design is denoted by $\Dmc_{5}$ in \cite{a:Zhou-sym-sporadic}.\smallskip  

\noindent \textbf{Line 7.} This design also arose from studying rank $3$ automorphism groups of symmetric designs \cite{a:Dempwolff2001}, and it has parameter set $(14080, 12636, 11340)$ admitting $\Fi_{22}$ as flag-transitive and point-primitive automorphism group with the point-stabiliser and block-stabiliser isomorphic to $\POm_7(3)$ (non-conjugate). This design is denoted by $\Dmc_{_{6}}$ in \cite{a:Zhou-sym-sporadic}. 

\section{Preliminary}\label{sec:pre}

In this section, we state some useful facts in both design theory and group theory, and then we establish our method in Section \ref{sec:method} in order to analyse symmetric designs admitting sporadic simple groups as their flag-transitive automorphism groups. 

\begin{lemma}\label{lem:six}{\rm \cite[Lemma 2.1]{a:ABD-PSL2}}
Let $\Dmc$ be a symmetric $(v,k,\lambda)$ design, and let $G$ be a flag-transitive automorphism group of $\Dmc$. If $\alpha$ is a point of $\Dmc$ and $H=G_{\alpha}$, then
\begin{enumerate}[\rm \quad (a)]
\item $k(k-1)=\lambda(v-1)$;
\item $k\mid |H|$ and $\lambda v<k^2$;
\item $k\mid \lambda\gcd(v-1,|H|)$;
\item $|G|< |H|^{3}$;
\item $k\mid \lambda e$, for all nontrivial subdegrees $e$ of $G$.
\end{enumerate}
\end{lemma}

Lemma~\ref{lem:comp} below is an elementary result that is useful in computational matter for finding potential parameter sets. An algorithm based on this lemma is given in \cite[Remark 3.8, Algorithm 1]{a:ABD-Exp}. We remark here that in \cite[Algorithm 1]{a:ABD-Exp}, we need to substitute the parameter ``$d$'' with parameter ``$t$'' as in this paper ``$d$'' is the number of the imprimitivity classes.

\begin{lemma}\label{lem:comp}\cite[Lemma~3.7]{a:ABD-Exp}
Let $\Dmc$ be a symmetric $(v,k,\lambda)$ design $\Dmc$, and let $mk=\lambda t$, where $t$ is a divisor of $v-1$, for some positive integer $m$. Then 
\begin{enumerate}[{\rm  (a)}]
\item $m\mid (k-1)$, and so $\gcd(m,k)=1$;
\item $\lambda=\lambda_{1}\lambda_{2}$, where $\lambda_{1}=\gcd(\lambda,k-1)$ and $\lambda_{2}=\gcd(\lambda,k)$;
\item If $k_{1}:=(k-1)/\lambda_{1}$ and $k_{2}:=k/\lambda_{2}$, then $v-1=k_{1}k_{2}$, where $k_{2}$ divides $t$. Moreover, $\lambda_{1}$ divides $m$,
$\lambda_{1}< k_{2}$ and $\gcd(\lambda_{1},k_{2})=1$.
\end{enumerate}
\end{lemma}

\begin{lemma}\cite[Theorem 1.1]{a:Praeger-imprimitive}\label{lem:imp}
Let $D = (\Pmc, \Bmc)$ be a nontrivial symmetric $(v,k,\lambda)$ design admitting a flag-transitive and point-imprimitive subgroup of automorphisms $G$ that leaves invariant a nontrivial partition $C$ of $\Pmc$ with $d$ classes of size $c$. Then there is a constant $\ell$ such that, for each $B\in \Bmc$ and $\Delta \in C$, the size $|B\cap \Delta|$ is either $0$ or $\ell$, and one of the following holds:
\begin{enumerate}[{\rm (a)}]
\item $k \leq \lambda(\lambda -3)/2$;
\item $(v, k, \lambda) = (\lambda^2(\lambda +2), \lambda(\lambda + 1), \lambda)$ with $(c, d, \ell) = (\lambda^2,\lambda +2,\lambda)$ or $(\lambda +2,\lambda^2, 2)$;
\item $(v, k,\lambda, c, d, \ell) = ( \frac{(\lambda+2)(\lambda^2-2\lambda+2)}{4} , \frac{\lambda^2}{2} , \lambda, \frac{\lambda+2}{2} , \frac{\lambda^2-2\lambda+2}{2} , 2)$, and either $\lambda \equiv 0 \mod 4$, or $\lambda = 2u^2$, where $u$ is odd, $u\geq 3$, and $2(u^2-1)$ is a square;
\item $(v, k, \lambda, c, d, \ell) = (\frac{(\lambda + 6)(\lambda^2+4\lambda-1)}{4} , \frac{\lambda (\lambda+5)}{2} , \lambda, \lambda+6, \frac{\lambda^2+4\lambda-1}{4},3)$, where $\lambda\equiv 1$ or $3$ $\mod 6$.
\end{enumerate}
\end{lemma}

With the same notation as in Lemma~\ref{lem:imp}, we have the following equalities:
\begin{align}
v&=cd,\label{eq:v-cd} \\
k&=\ell s, \label{eq:k-l} \\
\lambda(c-1)&=k(\ell-1),\label{eq:lam-1}
\end{align}
where $s$ is the number of classes of $\Cmc$ that intersect a block $B$ nontrivially, and $\ell>1$ and $s>1$.

\begin{lemma}\label{lem:m}
Let $\Dmc =(\Pmc,\Bmc)$ be a nontrivial symmetric $(v,k,\lambda)$ design admitting a flag-transitive and point-imprimitive automorphism group $G$, and let $(\alpha,B)$ be a flag of $\Dmc$. Let  also $M$ and $N$ be the maximal subgroups of $G$ containing the  point-stabiliser $H:=G_{\alpha}$ and the  block-stabiliser $K:=G_{B}$, respectively. Then
\begin{enumerate}[\rm (a)]
\item $v=|G:H|=|G:K|$;
\item $|G:M|\neq v$, and both $|G:M|$ and $|G:N|$ divide $v$;
\item $k=|H:H_{B}|$ and the index of a maximal subgroup of $H$ containing $H_{B}$ divides $k$;
\item $k=|K:K_{\alpha}|$ and the index of a maximal subgroup of $K$ containing $K_{\alpha}$ divides $k$. 
\end{enumerate} 
\end{lemma}
\begin{proof}
Part (a) follows from the fact that $G$ is point-transitive and block-transitive. The fact that $G$ is point-imprimitive implies that $M\neq H$, or equivalently, $|G:M|\neq |G:H|=v$, and so part (b) holds. Since $G$ is flag-transitive, $H$ is transitive on the set of blocks containing $\alpha$ and $K$ is transitive on $B$, and hence parts (c)-(d) follow.  
\end{proof}

\section{Method}\label{sec:method}

In this section, we  describe our method to find candidate parameters which are listed in Tables \ref{tbl:poss-A}-\ref{tbl:poss-B} in Appendix. In what follows, we suppose that $\Dmc = (\Pmc, \Bmc)$ is a nontrivial symmetric $(v,k,\lambda)$ design admitting a flag-transitive and point-imprimitive subgroup $G$ of automorphisms of $\Dmc$. Suppose also that $G$ is an almost simple group whose socle is a sporadic simple group. Let $(\alpha,B)$ be a flag of $\Dmc$, and let  $M$ and $N$ be the maximal subgroups of $G$ containing the  point-stabiliser $H:=G_{\alpha}$ and the  block-stabiliser $K:=G_{B}$, respectively. The possibilities for $M$ and $N$ can be read of from GAP \cite{GAP4} and Atlas \cite{b:Atlas}. We use the following approach to obtain the possible parameter sets.

\begin{enumerate}[\rm \bf(1)]
\item By Lemma~\ref{lem:m}(a), we have that $|H|=|K|$, and since $G$ is flag-transitive and $M$ and $N$ are respectively maximal subgroups containing $H$ and $K$, it follows that both $M$ and $N$ satisfy the large property as in Lemma~\ref{lem:six}(d), that is to say, $|G|\leq |M|^3$ and $|G|\leq |N|^{3}$. This enables us to obtain the  possibilities for both $M$ and $N$ by \cite[Proposition~6.2]{a:AB-Large-15}. Here, we also use the database stored in GAP \cite{GAP4}. 

\item For a fixed $M$ obtained in (1), since $v=|G:H|=|G:M|\cdot |M:H|$, we conclude that  $v=z\cdot|G:M|$, for some divisor $z$ of $|M|$. Therefore, for each $M$, we can find possible parameters $v$.  
\item For each $M$ and each (associated) $v$ obtained in (2), we know by Lemma~\ref{lem:six}(b) that $k$ divides $|H|$, so does $|M|$. We now use Lemma~\ref{lem:comp}. Let $t:=\gcd(v-1,|M|)$. Then for each divisor $k_{2}$ of $t$, let $k_{1}:=(v-1)/k_{2}$. For $\lambda_{1}<k_{2}$ with $\gcd(\lambda_{1},k_{2})=1$, let $k:=1+k_{1}\lambda_{1}$ and $\lambda:=k(k-1)/(v-1)$. We check if $\lambda$ is an integer, $k$ divides $M$ and $\lambda v<k^{2}$. This produces a list of possible parameters $(v,k,\lambda)$ for each $M$. Here we use modified version of \cite[Algorithm 1]{a:ABD-Exp} with changing $\lambda_{1}\leq k_{2}/2$ to $\lambda_{1}\leq k_{2}$ and substituting the restrictive condition $2k<v$ to $k<v-1$. This is for finding potential complement parameter sets. 
    \item Since $k:=|K:K_{\alpha}|$ divides $|K|=|H|=|G|/v$, it follows that $k$ divides $|N|$, where $N$ is a maximal subgroup of $G$ containing $K$. Note also that $|G:N|$ divides $v=|G:K|$. This leaves $22$ groups with $766$ tuples $(M,N,(v,k,\lambda))$ listed in Tables \ref{tbl:poss-A}-\ref{tbl:poss-B} for further consideration. For each group, the number of such candidate parameter sets is recorded in Table \ref{tbl:stats}. As we are also interested in point-imprimitive action, for each parameter set $(v,k,\lambda)$, we can also find parameters $(c,d,\ell)$ using 
\eqref{eq:v-cd}, \eqref{eq:k-l} and \eqref{eq:lam-1}. By Lemma~\ref{lem:imp}, we also can associate one of the type ``a'', ``b'', ``c'' and ``d'' to the parameter set $(v,k,\lambda,c,d,\ell)$ according to the part of Lemma~\ref{lem:imp} that the parameter set $(v,k,\lambda,c,d,\ell)$ satisfies.  
    \item For each pairs $(M,N)$, we check if $M$ (respectively, $N$) has a subgroup of index $\i_{H}:=v/|G:M|$ (respectively, $\i_{K}:=v/|G:N|$). This suggests the  pairs $(H,K)$ of point-block stabilisers. In the case where $M$ is of large order, we look for a maximal subgroup of $M$ whose index divides $\i_{H}$. This gives the  possibilities for oversubgroups of $H$, and then we continue this process to find possible candidates for $H$. We do the same when $N$ is of large order.
\item For each subgroup $H$ obtained in (5), we consider the right coset action of $G$ on the  set of the right cosets of $H$ in $G$, and obtain the $H$-orbits and their lengths (subdegrees). Since $G$ is flag-transitive, we check if $k$ is a divisor of $\lambda e$, for all nontrivial subdegrees $e$ of $G$, see Lemma~\ref{lem:six}(e). This also eliminates some cases.  
\item For the remaining pairs $(H,K)$ and parameters $(v,k,\lambda)$ that have not been eliminated in (6), we consider the right coset action of $G$ on the  set of the right cosets of $H$ in $G$, and obtain the $K$-orbits and their lengths in this action. Since $G$ is flag-transitive, the length of one of the $K$-orbits has to be $k$. If such an orbit $B$ exists, then we check if $B$ is a base block for a symmetric design using the software package ``\verb|design|'' in GAP \cite{GAP4}.
\end{enumerate}

\begin{table}
    \scriptsize
    \caption{The number of parameter sets for some sporadic almost simple groups.}\label{tbl:stats}
    \begin{tabular}{lc|lc|lc|lc|lc}
        \noalign{\smallskip}\hline\noalign{\smallskip}
        Group & \# parameters &
        Group & \# parameters &
        Group & \# parameters &
        Group & \# parameters &
        Group & \# parameters \\
        \noalign{\smallskip}\hline\noalign{\smallskip}
        $\M_{11}$ & $3$ & 
        $\J_2$ & $0$ & 
        $\Suz{:}2$ & $51$ & 
        $\ON{:}2$ & $52$ &
        $\Fi_{24}'{:}2$ & $298$ \\ 
        $\M_{12}$ & $26$ & 
        $\J_{2}{:}2$ & $9$ & 
        $\McL$ & $2$ &  
        $\Co_{1}$ & $0$ & 
        $\HN$ & $0$ \\ 
        $\M_{12}{:}2$ & $11$ & 
        $\J_{3}$ & $0$ & 
        $\McL{:}2$ & $35$ & 
        $\Co_{2}$ & $0$ & 
        $\HN{:}2$ & $59$ \\ 
        $\M_{22}$ & $6$ & 
        $\J_{3}{:}2$ & $11$ &
        $\Ru$ & $0$ & 
        $\Co3$ & $1$ 
        & $\Th$ & $0$ \\
        $\M_{22}{:}2$ & $19$ & 
        $\J_{4}$ & $0$ & 
        $\He$ & $1$ & 
        $\Fi_{22}$ & $12$ & 
        $\B$ & $0$ \\
        $\M_{23}$ & $1$ & 
        $\HS$ & $29$ & 
        $\He{:}2$ & $34$ & 
        $\Fi_{22}{:}2$ & $57$ 
        & $\M$ & $0$ \\
        $\M_{24}$ & $4$ & 
        $\HS{:}2$ & $45$ & 
        $\Ly$ & $0$ & 
        $\Fi_{23}$ & $0$ & \\
        $\J_{1}$ & $0$ & 
        $\Suz$ & $0$ & 
        $\ON$ & $0$ & 
        $\Fi_{24}'$ & $0$ & \\
        \noalign{\smallskip}\hline\noalign{\smallskip}
    \end{tabular} 
\end{table}

\section{Proof of the main results}\label{sec:proof}

In this section, we prove Theorem \ref{thm:main} and Corollary~\ref{cor:main}. In order to prove Theorem \ref{thm:main}, we follow the steps described in Section \ref{sec:method}.\smallskip 

\noindent{\textbf{Proof of Theorem~\ref{thm:main}}} 
Suppose that $\Dmc = (\Pmc, \Bmc)$ is a nontrivial symmetric $(v,k,\lambda)$ design admitting a flag-transitive, point-imprimitive subgroup $G$ of automorphisms of $\Dmc$. Suppose also that $G$ is an almost simple group whose socle is a sporadic simple group. Let $(\alpha,B)$ be a flag of $\Dmc$, and let  $M$ and $N$ be the maximal subgroups of $G$ containing the  point-stabiliser $H:=G_{\alpha}$ and the  block-stabiliser $K:=G_{B}$, respectively. We apply the method described in Section \ref{sec:method}, and as it is mentioned in step 4, there are 22 groups with total number $766$ of parameters which possibly give rise to a $2$-design. We divide these groups in 2 classes:

\begin{enumerate}[\rm (A)]
\item $\M_{11}$, $\M_{12}$, $\M_{12}{:}2$, $\M_{22}$, $\M_{22}{:}2$, $\M_{23}$, $\M_{24}$, $\J_{2}{:}2$, $\J_{3}{:}2$, $\HS$, $\HS{:}2$, $\McL$, $\McL{:}2$, $\He$, $\He{:}2$, $\Co_{3}$; 
\item  $\Suz:2$, $\ON:2$, $\HN:2$, $\Fi_{22}$, $\Fi_{22}:2$, $\Fi_{24}':2$.
\end{enumerate}

\noindent \textbf{(A)} For the groups in class A, once we follow  the steps in Section \ref{sec:method}, all possibilities but $6$ cases can be ruled out in steps 5 or 6. We record all the possibilities for the groups in Class A in Table \ref{tbl:poss-A}. In the last column of Table \ref{tbl:poss-A}, we use the notation ``nsg'' or ``nsd'', or we leave it blank. We have $6$ possible parameters in the latter case which are also separately recorded in Table \ref{tbl:M12-poss}.

\begin{table}
\scriptsize 
\caption{ The remaining possible parameters for the groups in class A.}\label{tbl:M12-poss}
\begin{tabular}{lllllllllll}
\hline 
$G$ & $M$ & $N$ & $(\nr(M),\nr(N))$ &  $(\i_{H}, \i_{K})$ & $v$ & $k$ & $\lambda$ & $c$ & $d$ & $\ell$  \\
\hline  
$\M_{12}$ & $\M_{11}$ & $\M_{11}$ & 
$(1 , 1)$ & 
$(12 , 12)$ & 
$144$ & $66$ & $30$ & $12$ & $12$ & $6$  \\ 
$\M_{12}$ & $\M_{11}$ & $\M_{11}$ & 
$(1 , 2)$ & 
$(12 , 12)$ & 
$144$ & $66$ & $30$ & $12$ & $12$ & $6$ \\
$\M_{12}$ & $\M_{11}$ & $\PSL_2(11)$ & 
$(1 , 5)$ & 
$(12 , 1)$ & 
$144$ & $66$ & $30$ & $12$ & $12$ & $6$ \\ 
$\M_{12}$ & $\M_{11}$ & $\M_{11}$ & 
$(2 , 1)$ & 
$(12 , 12)$ & 
$144$ & $66$ & $30$ & $12$ & $12$ & $6$  \\ 
$\M_{12}$ & $\M_{11}$ & $\M_{11}$ & 
$(2 , 2)$ & 
$(12 , 12)$ & 
$144$ & $66$ & $30$ & $12$ & $12$ & $6$ \\ 
$\M_{12}$ & $\M_{11}$ & $\PSL_2(11)$ & 
$(2 , 5)$ & 
$(12 , 1)$ & 
$144$ & $66$ & $30$ & $12$ & $12$ & $6$  \\ 
\hline 
\multicolumn{11}{l}{\scriptsize $\i_{H}:=|M:H|$ and $\i_{K}:=|N:K|$.}\\ 
\end{tabular}
\end{table}

The notation ``nsg'' (no subgroup) in Table \ref{tbl:poss-A} indicates that the case can be ruled out in step 6, that is to say, the maximal subgroup $M$ (or $N$) has no subgroup of index $\i_H$ (or $\i_K$). For example, if $G=\M_{12}{:}2$, then the point-stabiliser $H$ is contained in $M=\M_{12}$. Then $\M_{12}$ should have a subgroup of index $v/2$, where $v\in\{16, 22, 40, 66, 88, 96, 144, 160, 396\}$, which is impossible.

If ``nsd'' (no subdegree) is recorded in the last column of Table \ref{tbl:poss-A}, then it means that there exist some subgroup $H$ and $K$ with $|M:H|=\i_{H}$ and $|N:K|=\i_{K}$. By considering $H$ as a subgroup of the permutation group $G$ on $v=|G:H|$ points, we observe that there exist some nontrivial subdegrees $e$ of $G$ such that $\lambda e$ is not divisible by $k$, which violates Lemma~\ref{lem:six}(e). In this case, we also present the smallest such a subdegree in this column. 
For example, let $G=\HS$ and $(v,k,\lambda)=(8800, 420,20)$, and $M=\PSU_3(5):2$  and $N=\A_{8}:2$. Then $M$ has a subgroup $H=\S_{7}$ of index $|M:H|=50$. Then the permutation group   $G$ on the set of right cosets of $H$ with $8800$ points has nontrivial subdegrees $e\in\{7$, $42$, $126$, $210$, $252$, $630$, $1260$, $2520\}$, and it is easy to see that $k=420$ does not divide $\lambda e=20e$ for any such $e$, which is a contradiction.

This leaves  the possibilities recorded in Table \ref{tbl:M12-poss} in which  $G=\M_{12}$. Moreover, $(v,k,\lambda)=(144,66,30)$ for which the point-stabiliser $H$ is a subgroup of the maximal subgroup $M_{1}$ or $M_{2}$, where $M_{i}\cong \M_{11}$, and the block-stabiliser $K$ is a subgroup of the maximal subgroup $N_1\cong \M_{11}$, $N_{2}\cong \M_{11}$ or $N_{3}\cong \PSL_{2}(11)$. We note here that $M_{1}$ and $M_{2}$ (respectively, $N_{1}$ and $N_{2}$) are not conjugate in $G$.  It is also easy to check that $H=M_{1}\cap M_{2}$  and it is isomorphic to $\PSL_{2}(11)$. If $K\leq N_{3}$, then since $|G:N_{3}|=144=|G:K|$, we conclude in this case that $K=N_{3}\cong \PSL_{2}(11)$ is maximal in $G$, and the $K$-orbits are of length $12$ and $132$, which is a contradiction as $k=66$.

Suppose that $K\leq N_{1}$. Then $K=N_{1}\cap N_{2}$ which is conjugate to $H\cong \PSL_{2}(11)$. Thus $K\cong \PSL_{2}(11)$, and the length of $K$-orbits in the right coset action of $G$ on the set of right cosets of $H$ in $G$ are $11$, $11$, $55$ and $66$. The $K$-orbit $B$ of length $66$ gives rise to the symmetric $(144,66,30)$ design $\Dmc=(\Pmc,B^{G})$, where $G=\M_{12}$, $\Pmc=\{1,\ldots,144\}$ and $B$ is the base block
\begin{align}\label{eq:D-M12-144}
\nonumber B:= \{&1, 3, 4, 5, 8, 11, 13, 14, 16, 19, 20, 22, 26, 27, 28, 32, 33, 36, 37, 38, 41, 42, 44,\\ 
\nonumber&48, 49, 
50, 51, 58, 59, 60, 64, 66, 68, 70, 71, 72, 73, 74, 75, 82, 83, 84, 98, 102,\\
\nonumber& 103, 105, 106, 108, 
109, 110, 112, 115, 116, 118, 121, 123, 126, 128, 129, 130, \\
&135, 136, 139, 141, 142, 143\}.
\end{align} 
We denote this design by $\Dmc_{1}$ which is the design in Line 1 of Table \ref{tbl:main}. We note here that the full automorphism group of $\Dmc_{1}$ is $\M_{12}:2$ which is point-primitive while $\M_{12}$  is point-imprimitive with imprimitive class representatives $\{1, 13, 35, 38, 57, 62, 81, 91, 103, 109, 128, 140\}$ and $\{1, 2, 3, 4, 5, 6, 7, 8, 9, 10, 11, 12\}$.
This shows that $(c,d,\ell)=(12,12,6)$. The case where $K\leq N_{2}$ can be treated similarly and we find the same design. In fact, each pair of designs obtained in this way (the orbit of length $66$ of the intersection of two non-conjugate maximal subgroups isomorphic to $\M_{11}$) are isomorphic. 
\smallskip
 
\noindent \textbf{(B)} For the groups in class B, the  possible parameters are listed in Table \ref{tbl:poss-B}. Here, we prove that none of these groups gives rise to a $2$-design. Indeed, for each group, we show that one of steps 5-6 fails, but for the computation matters, we need to check these cases separately.\smallskip 

Let $G=\Fi_{22}$. In this case, the maximal subgroup $M$ whose index divides $v$ is one of the two non-conjugate subgroups (isomorphic to) $\POm_7(3)$. Then $H$ is a subgroup of $M$ of index $14$, $40$ or $105$. However, $\POm_7(3)$ has no subgroup of such an index as the smallest index of the maximal subgroups of $\POm_7(3)$ is $351$, see \cite[p. 109]{b:Atlas}.\smallskip

Let now $G=\Fi_{22}{:}2$. Then we have $190$ possible parameter sets $(v,k,\lambda)$ and $H$ is contained in one of the two non-conjugate maximal subgroups of $G$ being isomorphic to $\Fi_{22}$. Let $M$ be one of these subgroups. Then $\i_{H}=|M:H|$ is $8$, $20$, $32$, $33$, $44$, $48$, $56$, $60$, $72$, $80$, $88$, $108$, $128$, $143$, $144$, $165$, $168$, $182$, $189$, $198$, $224$, $243$, $270$, $273$, $280$, $320$, $336$, $350$, $448$, $455$, $540$, $585$, $648$, $704$, $728$, $864$, $
1008$, $1320$, $2048$, $2800$, $2912$, $4032$, $4050$, $4752$, $6804$, $7040$, $7488$, $10368$, $11200$, $18480$, $25344$, $52800$, $81648$, $98560$, $281600$ or $739200$. By inspecting the indices of the maximal subgroups of $M\cong \Fi_{22}$ dividing $\i_{H}$, we observe that $H$ is contained in one of two non-conjugate maximal subgroups $\POm_7(3)$ of $M$ of index $7$ or $20$, and so $\POm_7(3)$ has to have a subgroup of index $7$ or $20$, which is impossible as the smallest index of the maximal subgroups of $\POm_7(3)$ is $351$ by \cite[p. 109]{b:Atlas}.

\begin{table}
\scriptsize 
\caption{ The remaining possible parameters for the groups in class B.}\label{tbl:poss-B-rem}
\begin{tabular}{llp{13.6cm}}
\hline 
$G$ & $M$ & \multicolumn{1}{c}{The indices of maximal subgroups of $M$ }  \\
\hline  
$\Suz:2$ &
$\Suz$ & 
1782, 22880, 32760, 135135, 232960, 370656, 405405, 926640, 1216215, 2358720, 3203200, 
10378368, 17297280, 39916800, 
57480192, 177914880 \\
$\ON:2$ &
$\ON$ &
122760, 2624832, 2857239, 17778376, 30968784, 42858585, 58183776, 182863296 \\
$\HN:2$ &
$\HN$ &
1140000, 1539000, 16500000, 74064375, 108345600, 136515456, 164587500, 263340000, 264515625, 364041216, 1436400000, 
2926000000, 4681600000 \\
$\Fi_{24}':2$ &
$\Fi_{24}'$ &
306936, 4860485028, 14081405184, 50177360142, 125168046080, 258870277120,
 2503413946215, 5686767482760, 7819305288795, 
70094105804800, 91122337546240, 100087107696576, 155717756992512, 633363728392395, 8212275503308800, 57650174033227776, 
2306006961329111040, 11859464372549713920, 103770313259809996800, 574727888823563059200, 3091639677809511628800 \\
\hline 
\end{tabular}
\end{table}

For the remaining groups, we make the same argument. The required information about this groups are given in Table \ref{tbl:poss-B-rem}. If $G$ is one of these groups as in the first column of Table~\ref{tbl:poss-B-rem}, then $H$ is contained in the maximal subgroup $M$ of $G$ (indeed the socle of $G$) as in the second column of the same table, and so  and the index of a maximal subgroup of $M$ containing $H$ (listed in the third column) has to divide $\i_H$ which is recorded in Table~\ref{tbl:poss-B}, which is impossible.\smallskip

\noindent{\textbf{Proof of Corollary~\ref{cor:main}}} 
If $G$ is point-imprimitive, then by Theorem~\ref{thm:main}, we obtain the design in Line 1 of Table \ref{tbl:main}. If $G$ is point-primitive, then by
\cite[Theorem~1]{a:Zhou-sym-sporadic}, we obtain one of the designs in Lines 2-7 of Table \ref{tbl:main}.

\section*{Acknowledgements}

The authors are grateful to Alice Devillers and Cheryl E. Praeger for supporting their visit to UWA (The University of Western Australia) during February–June 2023. They also thank Bu-Ali Sina University for the support during their sabbatical leave. The authors thank the reviewers for their comments and suggestions.

\bibliographystyle{elsart-num-sort}


\clearpage

\section*{Appendix}

In this section, we present two tables namely Tables \ref{tbl:poss-A} and \ref{tbl:poss-B} which include all the possible parameter sets and the required information about the groups and their subgroups which we obtained by applying our method described in Section~\ref{sec:method}. In these tables, the first column shows the almost simple group $G$ with socle a sporadic simple group. Recall that $H$ and $K$ are our potential point-stabiliser and block-stabiliser, respectively. The subgroups $M$ and $N$ in the second and third columns are the maximal subgroups of $G$ containing $H$ and $K$, respectively. In the fourth column, the first coordinate $\nr(M)$ indicates the number associated to the maximal subgroup $M$ of $G$ in the library of software package ``\verb|AtlasRep|'' in GAP \cite{GAP4}, and similarly, the second coordinate $\nr(N)$ shows the number associated to the subgroup $N$. In column 5, $\i_{H}=|M:H|$ and $\i_{K}=|N:K|$. The candidates of parameters $v$, $k$ and $\lambda$ are recorded in columns 6-8. In column 9, we give the type of imprimitivity of group $G$ which is one of the parts (a)-(d) in Lemma~\ref{lem:imp} that the  parameters $(v,k,\lambda)$ satisfy. In last column of Table \ref{tbl:poss-A}, we have two notation ``nsg'' and ``nsd'', or it is left blank. The notation ``nsg'' (no subgroup) indicates that the maximal subgroup $M$ (or $N$) has no subgroup of index $\i_H$ (or $\i_K$). The notation  ``nsd'' (no subdegree) means that there exist some subgroup $H$, but $H$ as permutation subgroup of $G$ on $v$ points does not satisfy Lemma\ref{lem:six}(e), that is to say, there exist a non-trivial $H$-orbit of length $e$ such that $\lambda e$ is not divisible by $k$,  and in such a case, the smallest subdegree $e$ is presented. If the last column of a row in Table \ref{tbl:poss-A} is left blank, then the case has been treated separately in Section \ref{sec:proof}.  

\scriptsize
\begin{longtable}{lllcclllcl} 
\caption{ The possible parameters for groups in class A.}\label{tbl:poss-A}  \\ 
\hline 
$G$ & $M$ & $N$ & $(\nr(M),\nr(N))$ &  $(\i_{H}, \i_{K})$ & $v$ & $k$ & $\lambda$ & Type & Comments\\
\hline 
\endfirsthead 
\multicolumn{10}{c}{{\normalsize \scshape \tablename\ \thetable{}} - continued} \\ \\ 
\hline 
$G$ & $M$ & $N$ & $(\nr(M),\nr(N))$ &  $(\i_{H}, \i_{K})$ & $v$ & $k$ & $\lambda$ &  Type & Comments\\\hline 
\endhead 
\hline 
\multicolumn{10}{r}{{Continued}}\endfoot
\endlastfoot 
$\M_{11}$ & $\A_6.2_3$ & $\A_5.2$ & 
$(1 , 4)$ & 
$(6 , 1)$ & 
$66$ & $40$ & $24$ & a & nsg \\ 
$\M_{11}$ & $\A_6.2_3$ & $\A_6.2_3$ & 
$(1 , 1)$ & 
$(6 , 6)$ & 
$66$ & $40$ & $24$ & a & nsg \\ 
$\M_{11}$ & $\A_6.2_3$ & $\A_6.2_3$ & 
$(1 , 1)$ & 
$(8 , 8)$ & 
$88$ & $30$ & $10$ & a & nsg \\ 
$\M_{12}$ & $\M_{11}$ & $\PSL_2(11)$ & 
$(1 , 5)$ & 
$(12 , 1)$ & 
$144$ & $66$ & $30$ & a &  \\ 
$\M_{12}$ & $\M_{11}$ & $\M_{11}$ & 
$(1 , 1)$ & 
$(12 , 12)$ & 
$144$ & $66$ & $30$ & a &  \\ 
$\M_{12}$ & $\M_{11}$ & $\M_{11}$ & 
$(1 , 2)$ & 
$(12 , 12)$ & 
$144$ & $66$ & $30$ & a &  \\ 
$\M_{12}$ & $\M_{11}$ & $\PSL_2(11)$ & 
$(2 , 5)$ & 
$(12 , 1)$ & 
$144$ & $66$ & $30$ & a &  \\ 
$\M_{12}$ & $\M_{11}$ & $\M_{11}$ & 
$(2 , 1)$ & 
$(12 , 12)$ & 
$144$ & $66$ & $30$ & a &  \\
\hline  
\end{longtable}

\begin{longtable}{lllllllll} 
\caption{ The possible parameters for the groups in class B.}\label{tbl:poss-B}  \\ 
\hline 
$G$ & $M$ & $N$ & $(\nr(M),\nr(N))$ &  $(\i_{H}, \i_{K})$ & $v$ & $k$ & $\lambda$ &  Type \\
\hline 
\endfirsthead 
\multicolumn{9}{c}{{\normalsize \scshape \tablename\ \thetable{}} - continued} \\ \\ 
\hline 
$G$ & $M$ & $N$ & $(\nr(M),\nr(N))$ &  $(\i_{H}, \i_{K})$ & $v$ & $k$ & $\lambda$ & Type \\ \hline 
\endhead 
\hline 
\multicolumn{9}{r}{{Continued}}\endfoot 
\endlastfoot 
$\Suz{:}2$ & $\Suz$ & $\Suz$ & 
$(1 , 1)$ & 
$(8 , 8)$ & 
$16$ & $6$ & $2$ & b  \\ 
$\Suz{:}2$ & $\Suz$ & $\Suz$ & 
$(1 , 1)$ & 
$(20 , 20)$ & 
$40$ & $27$ & $18$ & a  \\ 
$\Suz{:}2$ & $\Suz$ & $\Suz$ & 
$(1 , 1)$ & 
$(32 , 32)$ & 
$64$ & $28$ & $12$ & a  \\ 
$\Suz{:}2$ & $\Suz$ & $\Suz$ & 
$(1 , 1)$ & 
$(33 , 33)$ & 
$66$ & $40$ & $24$ & a  \\ 
$\Suz{:}2$ & $\Suz$ & $\Suz$ & 
$(1 , 1)$ & 
$(44 , 44)$ & 
$88$ & $30$ & $10$ & a  \\ 
$\Suz{:}2$ & $\Suz$ & $\Suz$ & 
$(1 , 1)$ & 
$(48 , 48)$ & 
$96$ & $20$ & $4$ & b  \\ 
$\Suz{:}2$ & $\Suz$ & $\Suz$ & 
$(1 , 1)$ & 
$(56 , 56)$ & 
$112$ & $75$ & $50$ & a  \\ 
$\Suz{:}2$ & $\Suz$ & $\Suz$ & 
$(1 , 1)$ & 
$(60 , 60)$ & 
$120$ & $35$ & $10$ & a  \\
\hline 
\end{longtable}


\end{document}